\documentclass{amsart}
\usepackage{color}
   \def\blue{}
\newcommand{\aspas}[1]{``{#1}''}
\usepackage[utf8]{inputenc}
\usepackage[utf8]{inputenc}
\usepackage[T1]{fontenc}
\usepackage[all]{xy}
\usepackage{lipsum}
\usepackage{url}
\usepackage{stackrel}

\usepackage{amsmath,amsthm,amssymb}
\usepackage{tikz}
\usepackage{graphicx}

\theoremstyle{definition}
\newtheorem{theorem}{Theorem}[section]
\newtheorem{definition}[theorem]{Definition}
\newtheorem{example}[theorem]{Example}
\newtheorem{proposition}[theorem]{Proposition}
\newtheorem{lemma}[theorem]{Lemma}
\newtheorem{remark}[theorem]{Remark}
\newtheorem{corollary}[theorem]{Corollary}

\numberwithin{equation}{section}

\begin{document}


\vspace{0.5in}

\renewcommand{\bf}{\bfseries}
\renewcommand{\sc}{\scshape}
\vspace{0.5in}

\title[Sectional number and coincidence property]%
{Relative sectional number and the coincidence property \\ }

\author[C. A. I. Zapata---F. A. T. Estrella]{Cesar A. Ipanaque Zapata---Felipe A. Torres Estrella}
\address{\textsc{Cesar A. Ipanaque Zapata}\\
Departamento de Matem\'atica, IME\\
Universidade de S\~ao Paulo\\
Rua do Mat\~ao 1010 CEP: 05508-090 S\~ao Paulo-SP, Brazil}

\email{cesarzapata@usp.br}

\address{\textsc{Felipe A. Torres Estrella} Facultad de Ciencias Matemáticas - FCM-UNMSM\\ Universidad Nacional Mayor de San Marcos\\
Ciudad Universitaria - UNMSM, Av. República de Venezuela 3400, Cercado de Lima, Lima, Perú}
\email{felipeantony.torres@unmsm.edu.pe }


\thanks {The first author would like to thank grant\#2023/16525-7, grant\#2022/16695-7 and grant\#2022/03270-8, S\~{a}o Paulo Research Foundation (FAPESP) for financial support.}

\begin{abstract} For a Hausdorff space $Y$, a topological space $X$ and a map $g:X\to Y$, we present a connection between the relative sectional number of the first coordinate projection $\pi_{2,1}^Y:F(Y,2)\to Y$ with respect to $g$, and the coincidence property (CP) for $(X,Y;g)$, where $F(Y,2)$ stands for the ordered configuration space of $2$ distinct points on $Y$, and $(X,Y;g)$ has the coincidence  property (CP) if, for every map $f:X\to Y$, there is a  point $x$ of $X$ such that $f(x)=g(x)$. Explicitly, we demonstrate that $(X,Y;g)$ has the CP if and only if 2 is the minimal cardinality of open covers $\{U_i\}_{1\leq i\leq n}$ of $X$ such that each $U_i$ admits a local lifting for $g$ with respect to $\pi_{2,1}^Y$. This characterization connects a standard problem in coincidence theory to current research trends in sectional category and topological robotics. Motivated by this connection, we introduce the notion of relative topological complexity of a map. 
\end{abstract}

\maketitle


\section{Introduction, outline and main results}
In this article \aspas{space} means a topological space, and by a \aspas{map} we will always mean a continuous map. Also by \aspas{fibration} we will mean Hurewicz fibration. We write $\overline{y_0}$ to denote the constant map in $y_0$. 

\medskip For spaces $X$, $Y$  and a map $g:X\to Y$, we say that $(X,Y;g)$  has \textit{the coincidence  property} (CP) if, for every map $f:X\to Y$, there is a point $x$ of $X$ such that $f(x)=g(x)$. In \cite{Holsztynski1964} the author says that a map $g:X\to Y$ is \aspas{universal} (or, equivalently, \aspas{coincidence producing} in \cite{schirmer1967}) if for any map $f:X\to Y$ there exists $x\in X$ such that $f(x)=g(x)$. Hence, $g:X\to Y$ is universal means, in our terminology, that $(X,Y;g)$ has the coincidence property. The coincidence property is a generalization of the fixed point property. Recall that a space $X$ has  \textit{the fixed point property} (FPP) if, for every self-map $f:X\to X$, there is a point $x$ of $X$ such that $f(x)=x$. Note that $(X,X;1_X)$ has CP if and only if $X$ has FPP.

\medskip For FPP characterizations are known in restrictive categories. For instance, in 1969, Fadell proved the following statement for spaces in the category of connected compact metric ANRs (see \cite{fadell1970} for references):
If $X$ is a Wecken space, then $X$ has the FPP if and only if $N(f)\neq 0$ for every self-map $f:X\to X$ \cite[Theorem 4.1, p. 17]{fadell1970}. Furthermore, if $X$ is a Wecken space satisfying the Jiang condition, $J(X)=\pi_1(X)$, then $X$ has the FPP if and only if $L(f)\neq 0$ for every self-map $f:X\to X$ \cite[Theorem 4.2, p. 17]{fadell1970}. Here, \aspas{Wecken space} is used in the sense of Fadell \cite[Definition 3.2, p. 15]{fadell1970}.

\medskip Recently, C. A. I. Zapata and J.~González presented in \cite{zapata2020}, in the category of Hausdorff spaces and maps, a characterization of FPP in terms of the sectional number ($\text{sec}(-)$). In addition, K. Tanaka presents a similar characterization for finite spaces \cite{tanaka2024}. We address the natural question of whether (and how) CP can be characterized. 

\medskip Such characterizations are known when $Y=I^n$ is the $n$-cube. For instance:
\begin{itemize}
    \item When $X$ is a compact space of covering dimension at most $n$, then $(X,I^n;g)$ has CP if and only if the element $g^\ast(e)$ of the $n$th Čech cohomology group $H^n(X,g^{-1}(S^{n-1});\mathbb{Z})$ is different from $0$ for a generator $e^n$ of $H^n(I^n,S^{n-1};\mathbb{Z})$ (see \cite{Holsztynski1976}).
    \item When $X=I^n$, then $(I^n,I^n;g)$ has CP if and only if the restriction of $g$ to $g^{-1}(S^{n-1})$ is essential (see \cite{schirmer1983}). 
\end{itemize}
 
 \medskip In this work we characterize CP in terms of relative sectional number ($\text{sec}_-(-)$). In fact, virtually all properties developed in \cite{zapata2020} for the case of FPP are extended here to CP realm.

\medskip This characterization will be made in terms of natural projections between ordered configuration spaces. In more detail, let $X$ be a space and $k\geq 1$. The ordered configuration space of $k$ distinct points on $X$ (see \cite{fadell1962configuration}) is the space \[F(X,k)=\{(x_1,\ldots,x_k)\in X^k\mid ~~x_i\neq x_j\text{   whenever } i\neq j \},\] topologised as a subspace of the Cartesian power $X^k$. For $k\geq r\geq 1,$ there is a natural projection $\pi_{k,r}^X\colon F(X,k) \to F(X,r)$ given by $\pi_{k,r}^X(x_1,\ldots,x_r,\ldots,x_k)=(x_1,\ldots,x_r)$.

\medskip The study of relative sectional number for the map $\pi_{k,r}^Y$ is still non-existent and, in fact, this work takes a first step in this direction. In \cite[Theorem 3.14, p. 565]{zapata2020}, in the category of Hausdorff spaces and maps, the authors show that $X$ has FPP if and only if the sectional number $\mathrm{sec}\hspace{.1mm}(\pi_{2,1}^X)$ equals 2.  In this work, for $Y$ a Hausdorff space and $g:X\to Y$ a map, we demonstrate that $(X,Y;g)$ has CP if and only if the relative sectional number $\mathrm{sec}\hspace{.1mm}_g(\pi_{2,1}^Y)$ equals 2 (Theorem \ref{characterizacao-ppf}). Several examples are presented to illustrate the results arising in this field. As shown in Section \ref{tc-map}, a particularly interesting feature of our characterization comes from its connection to current research trends in topological robotics. For this purpose, we introduce the notion of relative topological complexity of a map. 

\medskip The paper is organized as follows: In Section~\ref{sn}, we recall the notion of sectional category (secat(-)), sectional number (sec(-)) (Definition~\ref{sec-secat-def}), relative sectional category ($\text{secat}_-(-)$), relative sectional number ($\text{sec}_-(-)$) (Definition~\ref{relative-sec-secat-defn}) and basic results about these numerical invariants (Proposition~\ref{basic-properties} and Proposition~\ref{basic-properties-relative}). We study the relative sectional number for the projection map $\pi_{k,1}^Y$ with respect to a map $g:X\to Y$. Lemma~\ref{secop-pi-k-1} presents an upper bound, $\mathrm{sec}\hspace{.1mm}_g(\pi_{k,1}^Y)\leq k$. In Section~\ref{sec:cp}, our goal is to study the notion of coincidence property and its connection with sectional theory. We establish and demonstrate our Main Theorem,  which states that $(X,Y;g)$ has CP if and only if the relative sectional number $\mathrm{sec}\hspace{.1mm}_g(\pi_{2,1}^Y)$ equals 2 (Theorem \ref{characterizacao-ppf}). Proposition~\ref{conditions-fpp-implies-cp} presents some conditions such that FPP implies CP. Proposition~\ref{cohomology-cp} presents a condition in terms of cohomology to obtain CP. In Section \ref{tc-map}, we introduce and study the notion of relative topological complexity ($\text{TC}_-(-)$) of a map (Definition~\ref{defn:relative-tc}) and basic results about this numerical invariant. Theorem~\ref{relatice-tc-relatice-sec} presents lower and upper bounds for the relative topological complexity. Proposition~\ref{prop:cp-implies-value-relative-tc} shows that  CP can be characterized in terms of relative topological complexity. Proposition~\ref{cp-relative-tc} shows that if $(X,Y;g)$ has CP, then the inequalities $2\leq\mathrm{TC}_g(\pi_{2,1}^Y)\leq\min\{\mathrm{TC}(\pi_{2,1}^Y),2\mathrm{TC}(F(Y,2))\}$ hold. 

\section{Sectional theory}\label{sn}
We shall follow the terminology from \cite{zapata2020} and  recall the notion of sectional category and sectional number of a map. Also, we recall from \cite{hopf-inv} the notion of relative sectional category and from \cite{zapata2022} the notion of relative sectional number. If $f$ is homotopic to $g$ we shall denote by $f\simeq g$. The map $1_Z:Z\to Z$ denotes the identity map. 

Let $f:X\to Y$ be a map.  A \textit{(homotopy) cross-section} or \textit{section} of $f$ is a (homotopy) right inverse of $f$, i.e., a map $s:Y\to X$, such that $f\circ s = 1_Y$ ($f\circ s \simeq 1_Y$). Moreover, given a subspace $U\subset Y$, a \textit{(homotopy) local section} of $f$ over $U$ is a (homotopy) section of the restriction map $f_|:f^{-1}(U)\to U$, i.e., a map $s:U\to X$, such that $f\circ s$ is (homotopic to) the inclusion $\mathrm{incl}_U:U\hookrightarrow Y$.

\subsection{Sectional number} 
We recall the following definitions (we shall follow the terminology from \cite{zapata2020}).
\begin{definition}\label{sec-secat-def} Let $f\colon X\to Y$ be a map. 
\begin{enumerate}
    \item The \textit{sectional number} of $f$, $\mathrm{sec}\hspace{.1mm}(f)$, is the minimal cardinality of open covers of $Y$, such that each element of the cover admits a local section to $f$. 
    \item The \textit{sectional category} of $f$, also called Schwarz genus of $f$ \cite{schwarz1958genus}, and denoted by $\mathrm{secat}(f),$ is the minimal cardinality of open covers of $Y$, such that each element of the cover admits  a  homotopy local  section to $f$.
\end{enumerate}
\end{definition}

Note that $\mathrm{secat}\hspace{.1mm}(f)\leq \mathrm{sec}\hspace{.1mm}(f)$. Furthermore, if $f$ is a fibration, then $\mathrm{secat}\hspace{.1mm}(f)=\mathrm{sec}\hspace{.1mm}(f)$.

\medskip From \cite{zapata2022}, we recall that a \textit{quasi pullback} means a commutative diagram \begin{eqnarray*}
\xymatrix{ X^\prime \ar[r]^{\psi} \ar[d]_{f^\prime} & X \ar[d]^{f} & \\
       Y^\prime  \ar[r]_{\,\,\varphi} &  Y &}
\end{eqnarray*} such that for any space $Z$ and any maps $\alpha:Z\to Y^\prime$ and $\beta:Z\to X$ satisfying $\varphi\circ\alpha=f\circ\beta$,
\begin{eqnarray*}
\xymatrix{ Z \ar@/_10pt/[ddr]_{\alpha} \ar@/^10pt/[rrd]^{\beta}&  & \\
&X^\prime \ar[r]^{\psi} \ar[d]_{f^\prime} & X \ar[d]^{f} & \\
       &Y^\prime  \ar[r]_{\,\,\varphi} &  Y &}
\end{eqnarray*}
there exists a (not necessarily unique) map $h:Z\to X^\prime$
\begin{eqnarray*}
\xymatrix{ Z \ar@/_10pt/[ddr]_{\alpha} \ar@/^10pt/[rrd]^{\beta}\ar@{-->}[rd]^{h}&  & \\
&X^\prime \ar[r]^{\psi} \ar[d]_{f^\prime} & X \ar[d]^{f} & \\
       &Y^\prime  \ar[r]_{\,\, \varphi} &  Y &}
\end{eqnarray*}such that $f^\prime\circ h=\alpha$ and $\psi\circ h=\beta$.

Recall the pathspace construction from \cite[p. 407]{hatcheralgebraic}. For a map $f:X\to Y$, consider the space 
\begin{equation*}
E_f=\{(x,\gamma)\in X\times PY\mid~\gamma(0)=f(x)\},
\end{equation*}  where $PY=Y^I$ is the space of all paths $[0,1]\to Y$ equipped with the compact-open topology. The map \begin{equation*}
\rho_f:E_f\to Y,\qquad (x,\gamma)\mapsto \rho_f(x,\gamma)=\gamma(1),
\end{equation*} is a fibration. Furthermore, the projection onto the first coordinate $E_f\to X,~(x,\gamma)\mapsto x$ is a homotopy equivalence with homotopy inverse $c:X\to E_f$ given by $x\mapsto (x,\overline{f(x)})$, where $\overline{f(x)}$ is the constant path at $f(x)$. This renders the factorization $$\left(X\stackrel{f}\to Y\right)\,=\,\left(X\stackrel{c}{\to} E_f\stackrel{\rho_f}{\to} Y\right),$$ a composition of a homotopy equivalence followed by a fibration.  

\medskip We also recall the notion of LS category which, in our setting, is the
one given in \cite{cornea2003lusternik} plus one. The \textit{Lusternik-Schnirelmann category} (LS category) or \textit{category} of a space $X$, denoted by cat$(X)$, is the least integer $m$ such that $X$ can be covered by $m$ open sets, all of which are contractible within $X$. We have $\text{cat}(X)=1$ iff $X$ is contractible. The LS category is a homotopy invariant, i.e., if $X$ is homotopy equivalent to $Y$ (which we shall denote by $X\simeq Y$), then $\text{cat}(X)=\text{cat}(Y)$. 

\medskip For convenience, we record the following standard statements, most of which appear in chapter~4 of~\cite{zapata2022thesis}:
\begin{proposition}\label{basic-properties}
\noindent
\begin{enumerate}
\item  If the following diagram of spaces and maps between them 
\begin{eqnarray*}
\xymatrix{ X^\prime \ar[rr]^{\,\,} \ar[dr]_{f^\prime} & & X \ar[dl]^{f}  \\
        &  Y & }
\end{eqnarray*}
commutes, then $\mathrm{sec}\hspace{.1mm}(f^\prime)\geq \mathrm{sec}\hspace{.1mm}(f)$ and $\mathrm{secat}\hspace{.1mm}(f^\prime)\geq \mathrm{secat}\hspace{.1mm}(f)$.
\item  If the following diagram of spaces and maps between them  
\begin{eqnarray*}
\xymatrix{ X^\prime \ar[rr]^{\,\,} \ar[dr]_{f^\prime} & & X \ar[dl]^{f}  \\
        &  Y & }
\end{eqnarray*}
commutes up homotopy, then $\mathrm{secat}\hspace{.1mm}(f^\prime)\geq \mathrm{secat}\hspace{.1mm}(f)$.
\item \cite[Lemma 2.2]{zapata2022} Let $f:X\to Y$ and $g:Z\to Y$ be maps. If the following square
\begin{eqnarray*}
\xymatrix{ W \ar[r]^{\,\,} \ar[d]_{f^\prime} & X \ar[d]^{f} & \\
       Z  \ar[r]_{\,\, g} &  Y &}
\end{eqnarray*}
is a quasi pullback, then $\mathrm{sec}\hspace{.1mm}(f^\prime)\leq \mathrm{sec}\hspace{.1mm}(f)$. 
\item For a map $f:X\to Y$, $\mathrm{secat}\hspace{.1mm}(f)=\mathrm{secat}\hspace{.1mm}(\rho_f)$.
\item If $f\simeq g$, then $\mathrm{secat}\hspace{.1mm}(f)=\mathrm{secat}\hspace{.1mm}(g).$
\item If $p:E\to B$ is a fibration, then $\mathrm{sec}\hspace{.1mm}(p)\leq \text{cat}(B)$. In particular, for any map $f:X\to Y$, we have $\mathrm{secat}(f)\leq \text{cat}(Y)$.
\item If $f:X\to Y$ is nulhomotopic, then $\mathrm{secat}(f)=\text{cat}(Y).$
\end{enumerate}
\end{proposition}

In addition, we establish and prove the following statement. 

\begin{proposition}\label{secatsecat-secat}
 Consider the following diagram of spaces and maps between them 
\begin{eqnarray*}
\xymatrix{  X \ar[r]^{\varphi} \ar[d]_{f} & X^\prime \ar[d]^{f^\prime}  \\
      Y  \ar[r]_{\psi} &  Y^\prime }. 
\end{eqnarray*}  
\begin{enumerate}
\item If $f^\prime\circ\varphi\simeq \psi\circ f$, then $\text{secat}(f)\cdot\text{secat}(\psi)\geq\text{secat}(f^\prime).$
    \item If $f^\prime\circ\varphi= \psi\circ f$, then $\text{sec}(f)\cdot\text{sec}(\psi)\geq\text{sec}(f^\prime)$.
    \end{enumerate}   
\end{proposition}
\begin{proof}
   Suppose $U\subset Y$, $V\subset Y'$, $\xi:V\to Y$ and take $W=\xi^{-1}(U)\subset Y^\prime$. Note that a map $\sigma:U\to X$ yields a map $\delta=\left(W\stackrel{\xi_|}{\to} U\stackrel{\sigma}{\to} X\stackrel{\varphi}{\to} X^\prime\right).$ 
If $f^\prime\circ\varphi= \psi\circ f$, $\psi\circ \xi=\mathrm{incl}_{V}$ and $f\circ\sigma=\mathrm{incl}_U$, then $f^\prime\circ\delta=\mathrm{incl}_W$. Therefore, we conclude that $\text{sec}(f)\cdot\text{sec}(\psi)\geq\text{sec}(f^\prime)$.

\medskip The proof for the case $f^\prime\circ\varphi\simeq \psi\circ f$ is analogous.
\end{proof}

Proposition~\ref{secatsecat-secat} generalizes \cite[Lemma 3.16]{zapata2022}. Proposition~\ref{basic-properties} together with Proposition~\ref{secatsecat-secat} implies the following statement. It was stated in \cite[p. 341]{james1978} for pullbacks. 

\begin{corollary}\label{fibration-secat-1}
 Consider the following quasi pullback
\begin{eqnarray*}
\xymatrix{ W \ar[r]^{\,\,} \ar[d]_{f^\prime} & X \ar[d]^{f}  \\
       Z  \ar[r]_{\,\, g} &  Y}.
\end{eqnarray*}
If $f$ is a fibration and $g:Z\to Y$ is a map such that $\mathrm{secat}\hspace{.1mm}(g)=1$ (for example when $g$ is a homotopy equivalence), then $\mathrm{sec}\hspace{.1mm}(f^\prime)=\mathrm{sec}\hspace{.1mm}(f)$.
\end{corollary}

Furthermore, Proposition~\ref{secatsecat-secat} implies the following statement.

\begin{proposition}\label{secat-sec-produc}
Let $f:X\to Y$ be a map and $Z$ be a space, then $\mathrm{sec}(1_Z\times f)=\mathrm{sec}(f)$ and $\mathrm{secat}(1_Z\times f)=\mathrm{secat}(f)$.
\end{proposition}
\begin{proof}
 Suppose $U\subset Y$. Note that a map $\sigma:U\to X$ yields a map $1_Z\times \sigma:Z\times U\to Z\times X$. If $f\circ\sigma=\mathrm{incl}_U$, then $\left(1_Z\times f\right) \circ\left(1_Z\times\sigma\right)=\mathrm{incl}_{Z\times U}$. Hence, $\mathrm{sec}(f)\geq \mathrm{sec}(1_Z\times f)$. The other inequality follows from the fact that the square 
\begin{eqnarray*}
\xymatrix@C=2cm{ Z\times X \ar[r]^{p_2} \ar[d]_{1_Z\times f} & X \ar[d]^{f}  \\
       Z\times Y \ar[r]_{p_2} &  Y }
\end{eqnarray*} where $p_2$ is the second coordinate projection, 
is commutative together with Item (2) from Proposition~\ref{secatsecat-secat} (here note that $\mathrm{sec}(p_2)=1$). Therefore, we conclude that $\mathrm{sec}(1_Z\times f)=\mathrm{sec}(f)$.

\medskip The other equality, $\mathrm{secat}(1_Z\times f)=\mathrm{secat}(f)$, is proved analogously.
\end{proof}

We have the following remark.

\begin{remark}
By Proposition~\ref{basic-properties}(3) we have that if the following square
\begin{eqnarray*}
\xymatrix{ W \ar[r]^{\,\,} \ar[d]_{f^\prime} & X \ar[d]^{f} & \\
       Z  \ar[r]_{\,\, } &  Y &}
\end{eqnarray*}
is a quasi pullback, then $\mathrm{sec}\hspace{.1mm}(f^\prime)\leq \mathrm{sec}\hspace{.1mm}(f)$. In contrast, the (corresponding) assertion ($\mathrm{secat}\hspace{.1mm}(f^\prime)\leq \mathrm{secat}\hspace{.1mm}(f)$) 
may fail. This can be seen by considering, for example, the constant map $\overline{1}:X\to\mathbb{R}^2$ and the inclusion $i:S^1\hookrightarrow \mathbb{R}^2$, \begin{eqnarray*}
\xymatrix{  & X \ar[d]^{\overline{1}}  \\
       S^1  \ar@{^{(}->}[r]_{\,\, i} &  \mathbb{R}^2 }.
\end{eqnarray*} Note that its canonical pullback is given by: 
\begin{eqnarray*}
\xymatrix{ X\ar[r]^{\,\,} \ar[d]_{\overline{1}} & X \ar[d]^{\overline{1}}  \\
       S^1 \ar@{^{(}->}[r]_{\,\, i} &  \mathbb{R}^2 }.
\end{eqnarray*} Hence, $\mathrm{secat}\hspace{.1mm}(\overline{1}:X\to S^1)=\mathrm{cat}(S^1)=2$ and $\mathrm{secat}\hspace{.1mm}(\overline{1}:X\to \mathbb{R}^2)=\mathrm{cat}(\mathbb{R}^2)=1$.
\end{remark}

\subsection{Relative sectional number}
Now, we recall the notion of relative sectional category together with basic results from \cite{hopf-inv} about this numerical invariant. Note that the notion of relative sectional category in González-Grant-Vandembroucq's paper \cite{hopf-inv} is given for a fibration. Here, the definition is stated for arbitrary continuous maps. Also, we record the notion of relative sectional number from \cite{zapata2022}.

\medskip For maps $p:E\to B$ and $g:X\to B$, \begin{eqnarray*}
\xymatrix{  & E \ar[d]^{p} & \\
       X  \ar[r]_{\,\, g} &  B &}
\end{eqnarray*} we consider the canonical pullback \[g^\ast(p):X\times_B E\to X,\] \begin{eqnarray*}
\xymatrix{ X\times_B E \ar[r]^{\,\,} \ar[d]_{g^\ast(p)} & E \ar[d]^{p} & \\
       X \ar[r]_{\,\, g} &  B &}
\end{eqnarray*} where $X\times_B E=\{(x,e)\in X\times E:~g(x)=p(e)\}$ and $g^\ast(p)(x,e)=x$. 

\medskip We have the following definitions.
\begin{definition}\label{relative-sec-secat-defn}
For maps $p:E\to B$ and $g:X\to B$. 
\begin{enumerate}
    \item The \textit{relative sectional number} $\mathrm{sec}\hspace{.1mm}_g(p)$, is the sectional number of the canonical pullback $g^\ast(p)$, i.e., \[\mathrm{sec}\hspace{.1mm}_g(p)=\mathrm{sec}\hspace{.1mm}(g^\ast(p)).\]
    \item The \textit{relative sectional category} $\mathrm{secat}\hspace{.1mm}_g(p)$, is the sectional category of the canonical pullback $g^\ast(p)$, i.e., \[\mathrm{secat}\hspace{.1mm}_g(p)=\mathrm{secat}\hspace{.1mm}(g^\ast(p)).\]
\end{enumerate}
\end{definition}

Note that $\mathrm{sec}\hspace{.1mm}_{1_B}(p)=\mathrm{sec}\hspace{.1mm}(p)$ and $\mathrm{secat}\hspace{.1mm}_{1_B}(p)=\mathrm{secat}\hspace{.1mm}(p)$ for any map $p:E\to B$. In addition, we have $\mathrm{secat}\hspace{.1mm}_g(p)\leq \mathrm{sec}\hspace{.1mm}_g(p)$ for any maps $p:E\to B$ and $g:X\to B$. Furthermore, if $p:E\to B$ is a fibration, then $\mathrm{sec}\hspace{.1mm}_g(p) = \mathrm{secat}\hspace{.1mm}_g(p)$ for any map $g:X\to B$. If $X\subset B$ and $g$ is the inclusion, note that $\mathrm{sec}\hspace{.1mm}_{g}(p)=\mathrm{sec}\hspace{.1mm}\left(p_|:p^{-1}(A)\to A\right)$ and $\mathrm{secat}\hspace{.1mm}_{g}(p)=\mathrm{secat}\hspace{.1mm}\left(p_|:p^{-1}(A)\to A\right)$.

\begin{remark}\cite[Remark 2.4]{zapata2022}\label{lifting}
Note that the number $\mathrm{sec}\hspace{.1mm}_g(p)$ coincides with the minimal cardinality of open covers $\{U_i\}_{1\leq i\leq n}$ of $X$, such that each element of the cover admits a lift $\sigma_i:U_i\to E$ to $g$, i.e., $p\circ \sigma_i=g_{\mid U_i}$. 
\end{remark}

Remark~\ref{lifting} implies the following 

\begin{proposition}\label{composition}
    Let $p_1:E_1\to E_2$, $p_2:E_2\to B$ and $g:X\to B$ be maps, 
    \begin{eqnarray*}
\xymatrix{ & E_1 \ar[d]^{p_1} \\
& E_2 \ar[d]^{p_2}  \\
     X  \ar[r]_{g} &  B}
\end{eqnarray*}
    then \[\mathrm{sec}\hspace{.1mm}_g(p_2)\leq\mathrm{sec}\hspace{.1mm}_g(p_2\circ p_1)\leq\mathrm{sec}\hspace{.1mm}_g(p_2)\cdot\mathrm{sec}\hspace{.1mm}(p_1).\]
\end{proposition}
\begin{proof}
  Suppose $U\subset X$. Note that a map $\sigma:U\to E_1$ yields a map $\delta=\left(U\stackrel{\sigma}{\to} E_1\stackrel{p_1}{\to} E_2\right).$ 
If $\left(p_2\circ p_1\right)\circ\sigma=g_{| U}$, then $p_2\circ\delta=g_{| U}$. Thus, $\mathrm{sec}\hspace{.1mm}_g(p_2\circ p_1)\geq \mathrm{sec}\hspace{.1mm}_g(p_2)$.

Now, suppose $U\subset X$, $\sigma:U\to E_2$ is a map and $V\subset E_2$. Take $W=\sigma^{-1}(V)\subset X$. Note that a map $s:V\to E_1$ yields a map $\rho=\left(W\stackrel{\sigma_|}{\to} V\stackrel{s}{\to} E_1\right)$. If $p_2\circ\sigma=g_{| U}$ and $p_1\circ s=\mathrm{incl}_V$, then $(p_2\circ p_1)\circ\rho=g_{| W}$. Hence, $\mathrm{sec}\hspace{.1mm}_g(p_2)\cdot\mathrm{sec}\hspace{.1mm}(p_1)\geq \mathrm{sec}\hspace{.1mm}_g(p_2\circ p_1)$.   
\end{proof}

\medskip From Proposition~\ref{basic-properties}, we have the following statements:

\begin{proposition}\label{basic-properties-relative}
Let $p:E\to B$ and $g:X\to B$ be maps.
\begin{enumerate}
\item Then $\mathrm{sec}_g\hspace{.1mm}(p)\leq\mathrm{sec}\hspace{.1mm}(p)$ and $\mathrm{secat}_{g}\hspace{.1mm}(p)\leq \text{cat}(X)$. 
\item If $g\simeq g'$, then $\mathrm{secat}_{g}\hspace{.1mm}(p)=\mathrm{secat}_{g'}\hspace{.1mm}(p)$.
\item Consider the following diagram 
\begin{eqnarray*}
\xymatrix{  X\times_B E \ar[r]^{} \ar[d]_{g^\ast(p)} & E \ar[d]^{p} & \\
       X  \ar[r]_{g} &  B &} 
\end{eqnarray*}  where the square is the canonical pullback. Then \[\mathrm{sec}_g\hspace{.1mm}(p)\cdot \mathrm{sec}\hspace{.1mm}(g)\geq \mathrm{sec}\hspace{.1mm}(p) \quad \text{ and }\quad \mathrm{secat}_g\hspace{.1mm}(p)\cdot\mathrm{secat}\hspace{.1mm}(g)\geq\mathrm{secat}\hspace{.1mm}(p).\]
    In particular, if $p$ is a fibration and $\mathrm{secat}\hspace{.1mm}(g)=1$ (for example when $g$ is a homotopy equivalence), then $\mathrm{sec}_g\hspace{.1mm}(p)=\mathrm{sec}\hspace{.1mm}(p)$. 
\end{enumerate}
\end{proposition}

One has the following example.

\begin{example}
 Let $B\subset X$. If $p:E\to B$ is a map and $r:X\to B$ is a retraction, then $\mathrm{sec}_r\hspace{.1mm}(p)=\mathrm{sec}\hspace{.1mm}(p)$. 
\end{example}

\subsection{Configuration spaces}\label{secconfespa}

Let $X$ be a space, and $k\geq 1$. The \textit{ordered configuration space} of $k$ distinct points on $X$ (see \cite{fadell1962configuration}) is the topological space \[F(X,k)=\{(x_1,\ldots,x_k)\in X^k\mid ~~x_i\neq x_j\text{   whenever } i\neq j \},\] topologised as a subspace of the Cartesian power $X^k$.

\medskip For $k\geq r\geq 1,$ there is a natural projection \blue{$\pi_{k,r}^X\colon F(X,k) \to F(X,r)$ given by $\pi_{k,r}^X(x_1,\ldots,x_r,\ldots,x_k)=(x_1,\ldots,x_r)$.}

\medskip By Proposition~\ref{basic-properties-relative}(1) together with \cite[Proposition 3.10, p. 565]{zapata2020} one has the following statement.

\begin{lemma}[Key lemma]\label{secop-pi-k-1}
 For any $k\geq 2$, a Hausdorff space $Y$, and any map $g:X\to Y$, we have
$\mathrm{sec}\hspace{.1mm}_g(\pi_{k,1}^Y)\leq k$.
\end{lemma}
\begin{proof}
By Proposition~\ref{basic-properties-relative}(1) we have that $\mathrm{sec}\hspace{.1mm}_g(\pi_{k,1}^Y)\leq \mathrm{sec}\hspace{.1mm}(\pi_{k,1}^Y)$. By \cite[Proposition 3.10, p. 565]{zapata2020}, we obtain $\mathrm{sec}\hspace{.1mm}(\pi_{k,1}^Y)\leq k$. Hence, $\mathrm{sec}\hspace{.1mm}_g(\pi_{k,1}^Y)\leq k$. 
\end{proof}

\section{Coincidence property}\label{sec:cp}

In this section, we present the coincidence property and its connections with sectional theory. 

\begin{definition}[Coincidence property]\label{defifpp}
For spaces $X$, $Y$  and a map $g:X\to Y$, we say that $(X,Y;g)$  has \textit{the coincidence  property} (CP) if, for every map $f:X\to Y$, there is a  point $x$ of $X$ such that $f(x)=g(x)$.
\end{definition}

As mentioned in the introduction, the coincidence property is a generalization of the fixed point property. Recall that a space $X$ has \textit{the fixed point property} (FPP) if, for every self-map $f:X\to X$, there is a  point $x$ of $X$ such that $f(x)=x$. Note that $(X,X;1_X)$ has CP if and only if $X$ has FPP.

\begin{example}\label{constant-has-not-cp}
Let $X, Y$ be spaces where $Y$ admits at least two elements. Recall that $\overline{y_0}:X\to Y$ denotes the constant map in $y_0\in Y$. One has $(X,Y;\overline{y_0})$ has not CP. 
\end{example}

We have the following statement, which says that CP implies FPP. 

\begin{proposition}\label{coincidence-implies-fpp}
 Let $X$, $Y$ be spaces, and $g:X\to Y$ be a map. If $(X,Y;g)$ has CP, then $Y$ has FPP.  
\end{proposition}
\begin{proof}
Let $f:Y\to Y$ be a self-map. We have a map $f\circ g:X\to Y$. Then, there is a point $x\in X$ such that $(f\circ g)(x)=g(x)$. Taking $y=g(x)$, one has $f(y)=y$. Thus, $Y$ has FPP.  
\end{proof}

The converse of Proposition~\ref{coincidence-implies-fpp} does not hold. For example, the unit disc $D^m=\{x\in\mathbb{R}^m:~\parallel x\parallel\leq 1\}$ has FPP (this is just the Brouwer's fixed point theorem) but, for any space $X$, the triple $(X,D^m;\overline{0})$ has not CP (see Example~\ref{constant-has-not-cp}).

\medskip We apply Proposition~\ref{coincidence-implies-fpp} in the case that $Y$ is a topological group. 

\begin{example}[Topological groups]\label{exam:lie-groups}
Note that no non-trivial topological group $G$ has the FPP. Thus, by Proposition~\ref{coincidence-implies-fpp}, $(X,G;g)$ has not CP, for any space $X$ and any map $g:X\to G$ (and, of course, $\mathrm{sec}\hspace{.1mm}_g(\pi_{2,1}^{G})=1$).
\end{example}

Also, we apply Proposition~\ref{coincidence-implies-fpp} in the case that $Y$ is a sphere $S^n$. 

\begin{example}[Spheres]\label{exam:spheres}
The spheres $S^n$ have not FPP. Thus, by Proposition~\ref{coincidence-implies-fpp}, $(X,S^n;g)$ has not CP, for any space $X$ and any map $g:X\to S^n$ (and, of course, $\mathrm{sec}\hspace{.1mm}_g(\pi_{2,1}^{S^n})=1$).
\end{example}

Also, we apply Proposition~\ref{coincidence-implies-fpp} in the case that $Y$ is a closed surface $\Sigma$ other than the projective plane,  $\Sigma\neq\mathbb{R}P^2$.

\begin{example}[Surfaces]\label{exam:surfaces}
We know that no closed surface $\Sigma$ (except for the projective plane, $\Sigma\neq\mathbb{R}P^2$) has the FPP. Thus, by Proposition~\ref{coincidence-implies-fpp}, $(X,\Sigma;g)$ has not CP, for any space $X$ and any map $g:X\to \Sigma$ (and, of course, $\mathrm{sec}\hspace{.1mm}_g(\pi_{2,1}^{\Sigma})=1$).
\end{example}

Before we present our Main Theorem we observe the following remark.

\begin{remark}\label{rem:nonCP-1}
  Let $X$, $Y$ be spaces and $g:X\to Y$ be a map. Note that the canonical pullback $g^\ast\left(\pi_{2,1}^Y\right):X\times_Y F(Y,2)\to X$ admits a cross-section if and only if there exists a map $f:X\to Y$ such that $f(x)\neq g(x)$ for any $x\in X$ (see Remark~\ref{lifting}). Hence, $\mathrm{sec}\hspace{.1mm}_g(\pi_{2,1}^Y)=1$ if and only if $(X,Y;g)$ has not CP.  
\end{remark}

 Thus, we have the following theorem.

\begin{theorem}[Main theorem]\label{characterizacao-ppf}
Let $X$, $Y$ be spaces where $Y$ is Hausdorff, and $g:X\to Y$ be a map. One has $(X,Y;g)$ has CP if and only if $\mathrm{sec}\hspace{.1mm}_g(\pi_{2,1}^Y)=2$.
\end{theorem}
\begin{proof}
Suppose that $(X,Y;g)$ has CP, then $\mathrm{sec}\hspace{.1mm}_g(\pi_{2,1}^Y)\geq 2$ (see Remark~\ref{rem:nonCP-1}) so, by the Key Lemma (Lemma~\ref{secop-pi-k-1}), $\mathrm{sec}\hspace{.1mm}_g(\pi_{2,1}^Y)=2$. Suppose now that $\mathrm{sec}\hspace{.1mm}_g(\pi_{2,1}^Y)=2$, so in particular $\mathrm{sec}\hspace{.1mm}_g(\pi_{2,1}^Y)\neq 1$. Hence, $(X,Y;g)$ has CP. 
\end{proof}

\medskip We recall the classical Fadell-Neuwirth fibration.  

\begin{lemma}[Fadell-Neuwirth fibration \cite{fadell1962configuration}] \label{TFN} Let $M$ be a connected $m$-dimensional topological manifold without boundary, where $m\geq 2$. For $k> r\geq 1$, the map $\pi_{k,r}^M:F(M,k)\to F(M,r)$  is a locally trivial bundle with fiber $F(M-Q_r, k-r)$, where $Q_r\subset M$ is a finite subset with $r$ elements. In particular, $\pi_{k,r}^M$ is a fibration.
\end{lemma}

 From Proposition~\ref{basic-properties-relative} (see also Corollary~\ref{fibration-secat-1}) we have that if $p$ is a fibration and $\mathrm{secat}\hspace{.1mm}(g)=1$, then $\mathrm{sec}_g\hspace{.1mm}(p)=\mathrm{sec}\hspace{.1mm}(p)$. The condition $p$ is a fibration can not be relaxed. For example, for $m\geq 1$, the projection $\pi_{2,1}^{D^m}:F(D^m,2)\to D^m$ is not a fibration, because the fibre $D^m-\{0\}$ is not homotopy equivalent to the fibre $D^m-\{(1,0,\ldots,0)\}$. Moreover, given any space $X$, we have that the constant map $\overline{0}:X\to D^m$ satisfies the equality $\mathrm{secat}\hspace{.1mm}(\overline{0})=1$. Furthermore, $\mathrm{sec}\hspace{.1mm}_{\overline{0}}(\pi_{2,1}^{D^m})=1$ (see Example~\ref{constant-has-not-cp}) and $\mathrm{sec}\hspace{.1mm}(\pi_{2,1}^{D^m})=2$ (see \cite[Theorem 3.14, p. 565]{zapata2020}). 

\medskip Now, we present some conditions such that the converse of Proposition~\ref{coincidence-implies-fpp} holds.

\begin{proposition}\label{conditions-fpp-implies-cp}
  Let $X$, $Y$ be spaces where $Y$ is a connected $m$-dimensional topological manifold without boundary, where $m\geq 2$. Let $g:X\to Y$ be a map such that $\mathrm{secat}\hspace{.1mm}(g)=1$ (for example when $g$ is a homotopy equivalence). If $Y$ has FPP, then $(X,Y;g)$ has CP (and, of course, $\mathrm{sec}\hspace{.1mm}_{g}(\pi_{2,1}^{Y})=2$). In particular, we obtain that $Y$ has FPP if and only if $(X,Y;g)$ has CP.  
\end{proposition}
\begin{proof}
 Suppose that $Y$ has FPP, then by \cite[Theorem 3.14, p. 565]{zapata2020}, $\mathrm{sec}\hspace{.1mm}(\pi_{2,1}^{Y})=2$. By Lemma~\ref{TFN}, the projection $\pi_{2,1}^{Y}:F(Y,2)\to Y$ is a fibration. Hence, by Proposition~\ref{basic-properties-relative}, $\mathrm{sec}\hspace{.1mm}_{g}(\pi_{2,1}^{Y})=\mathrm{sec}\hspace{.1mm}(\pi_{2,1}^{Y})$ (here we also use $\mathrm{secat}\hspace{.1mm}(g)=1$). Therefore, $\mathrm{sec}\hspace{.1mm}_{g}(\pi_{2,1}^{Y})=2$, and thus, by Theorem~\ref{characterizacao-ppf}, we conclude that $(X,Y;g)$ has CP.     
\end{proof}

We apply Proposition~\ref{conditions-fpp-implies-cp} to the case that $Y$ is an even-dimensional projective space $\mathbb{R}P^{2n}$. 

\begin{example}[Projective spaces]\label{exam:projective}
We recall that the odd-dimensional projective space $\mathbb{R}P^{2n+1}$ has not FPP. Thus, by Proposition~\ref{coincidence-implies-fpp}, $(X,\mathbb{R}P^{2n+1};g)$ has not CP, for any space $X$ and any map $g:X\to \mathbb{R}P^{2n+1}$ (and, of course, $\mathrm{sec}\hspace{.1mm}_g(\pi_{2,1}^{\mathbb{R}P^{2n+1}})=1$).

On the other hand, we know that even-dimensional projective spaces $\mathbb{R}P^{2n}$ have FPP. Thus, by Proposition~\ref{conditions-fpp-implies-cp}, $(X,\mathbb{R}P^{2n};g)$ has CP for any space $X$ and any map $g:X\to \mathbb{R}P^{2n}$ with $\mathrm{secat}(g)=1$ (and, of course,  $\mathrm{sec}\hspace{.1mm}_g(\pi_{2,1}^{\mathbb{R}P^{2n}})=2$). Analogous facts hold for complex and quaternionic projective spaces.
\end{example}

\subsection{Maps with codomain $D^m$} Given $m\geq 1$, $X$ a space and $g:X\to D^m$ a map, note that $(X,D^m;g)$ has not CP if and only if there is a map $\varphi:X\to S^{m-1}$ such that $\varphi(x)=g(x)$ for any $x\in g^{-1}(S^{m-1})$ (cf. \cite[Proposition 1.1, p. 148]{holsztynski1969}).
\begin{eqnarray*}
\xymatrix{ X\ar[rd]^{\,\,\varphi} &  & \\
       g^{-1}(S^{m-1}) \ar@{^{(}->}[u]_{i} \ar[r]_{\,\,g_|}&  S^{m-1} &}
\end{eqnarray*}
Indeed, suppose that $(X,D^m;g)$ has not CP, that is, there is a map $f:X\to D^m$ such that $f(x)\neq g(x)$ for any $x\in X$. Then, take $\varphi:X\to S^{m-1}$ given by $\varphi(x)=f(x)+t\left(g(x)-f(x)\right)$ where $t\geq 0$ such that $f(x)+t\left(g(x)-f(x)\right)\in S^{m-1}$. Now, suppose that there is a map $\varphi:X\to S^{m-1}$ such that $\varphi(x)=g(x)$ for any $x\in g^{-1}(S^{m-1})$. Take $f(x)=-\varphi(x)$ for any $x\in X$. Note that $f(x)\neq g(x)$ for any $x\in X$. Hence, $(X,D^m;g)$ has not CP. 

\medskip As mentioned at the beginning of the introduction, $g:X\to D^m$ being universal means, in our terminology, that $(X,D^m;g)$ has CP. Therefore, the following statement holds.

\begin{proposition}\label{cohomology-cp}(cf. \cite[Proposition 1.1, p. 148]{holsztynski1969})
  Let $m\geq 1$, $X$ be a space, and $g:X\to D^m$ be a map. If for a cohomology theory $h^\ast$ we have \begin{equation*}
 \left(g_|\right)^\ast(e^{m-1})\not\in\mathrm{Im}\left(i^\ast:h^\ast(X)\to h^\ast\left(g^{-1}(S^{m-1})\right)\right)   
\end{equation*} where $i:g^{-1}(S^{m-1})\hookrightarrow X$ is the inclusion map and $e^{m-1}\neq 0$ is an element of $h^{m-1}(S^{m-1})$, then $(X,D^m;g)$ has CP (and, of course, $\mathrm{sec}\hspace{.1mm}_g(\pi_{2,1}^{D^m})=2$).  
\end{proposition}
\begin{proof}
    Suppose that $(X,D^m;g)$ has not CP. Then, we have the following commutative diagram
 \begin{eqnarray*}
\xymatrix@C=3cm{ X\ar[rd]^{} &   \\
       g^{-1}(S^{m-1}) \ar@{^{(}->}[u]_{i} \ar[r]_{\,\,g_|}&  S^{m-1} }.
\end{eqnarray*} It implies the following commutative diagram  
\begin{eqnarray*}
\xymatrix@C=3cm{ h^\ast\left(X\right)\ar[d]_{i^\ast} &   \\
       h^\ast\left(g^{-1}(S^{m-1})\right) &  h^\ast\left(S^{m-1}\right) \ar[l]^{\left(g_|\right)^\ast} \ar[lu]^{} }.
\end{eqnarray*} Thus, $\left(g_|\right)^\ast(e^{m-1})\in\mathrm{Im}\left(i^\ast\right)$.
\end{proof}


\section{Relative topological complexity of a map}\label{tc-map}
In this section, we introduce and study the notion of relative topological complexity of a map, along with basic results about this numerical invariant. This is motivated by the coincidence property.

\medskip Recall that $Z^{[0,1]}$ denotes the space of all paths $\gamma: [0,1] \to Z$ in $Z$ and  $e^Z_{2}: Z^{[0,1]} \to Z \times Z$ denotes the fibration associating to any path $\gamma\in Z^{[0,1]}$ the pair of its initial and end points $e^Z_{2}(\gamma)=(\gamma(0),\gamma(1))$. Equip the path space $Z^{[0,1]}$ with the compact-open topology. 

Let $f:Z\to Y$ be a map, and let $e_f:Z^{[0,1]}\to Z\times Y,~e_f=(1_Z\times f)\circ e^Z_{2}$.
\begin{eqnarray*}
\xymatrix{ Z^{[0,1]} \ar[rd]^{e^Z_{2}}\ar[dd]_{e_f}  &  \\
& Z\times Z\ar[ld]_{1_Z\times f} \\
       Z\times Y  &  }
\end{eqnarray*}

\medskip For maps $f:Z\to Y$ and $g:X\to Y$, \begin{eqnarray*}
\xymatrix{  & Z \ar[d]^{f} & \\
       X  \ar[r]_{\,\, g} &  Y &}
\end{eqnarray*} we consider the canonical pullback 
\begin{eqnarray*}
\xymatrix{ \left(Z\times X\right)\times_{Z\times Y} Z^{[0,1]} \ar[r]^{\,\,} \ar[d]_{\left(1_Z\times g\right)^\ast(e_f)} & Z^{[0,1]} \ar[d]^{e_f} \\
       Z\times X \ar[r]_{1_Z\times g} &  Z\times Y}
\end{eqnarray*} 

Thus, we introduce the following definition.

\begin{definition}[Relative topological complexity of a map]\label{defn:relative-tc}
The \textit{topological complexity} of the map $f$ \textit{relative} to $g$, denoted by TC$_{g}(f)$, is the sectional number $\mathrm{sec}_{1_Z\times g}\hspace{.1mm}(e_{f})$ of the map $e_{f}$ relative to $1_Z\times g$. We can consider the homotopy version of this invariant. The \textit{homotopy topological complexity} of the map $f$ \textit{relative} to $g$, denoted by HTC$_{g}(f)$, is the sectional category $\text{secat}_{1_Z\times g}\hspace{.1mm}(e_{f})$ of the map $e_{f}$ relative to $1_Z\times g$.
\end{definition}

Since controlling the work map only up to homotopy makes the results harder to interpret in terms of actual motion plannings, one may use the strong notion rather than the homotopic one in the above definition. 

\begin{remark}
    The relative topological complexity $\mathrm{TC}_{1_Y}\hspace{.1mm}(f)$ coincides with the topological complexity of $f$, $\mathrm{TC}\hspace{.1mm}(f)$, in the sense of \cite{zapata2020}, \cite{pavesic2018} or \cite{pavesic2017}. Similarly, the homotopy relative topological complexity $\mathrm{HTC}_{1_Y}\hspace{.1mm}(f)$ coincides with the homotopy topological complexity of $f$, $\mathrm{HTC}\hspace{.1mm}(f)$, in the sense of \cite{zapata2020}. Recall that the topological complexity $\mathrm{TC}\hspace{.1mm}(f)$ is the sectional number $\mathrm{sec}\hspace{.1mm}(e_f)$ of the map $e_f$, and the homotopy topological complexity $\mathrm{HTC}\hspace{.1mm}(f)$ is the sectional category $\mathrm{secat}\hspace{.1mm}(e_f)$ of the map $e_f$. Furthermore, $\mathrm{TC}(1_X)$ is the famous Farber's topological complexity of $X$ which is denoted by $\mathrm{TC}(X)$, see \cite{farber2003topological}.
\end{remark}

Several different definitions of topological complexity of maps exist in the literature, for instance, \cite{pavesic2019}, \cite{murillo2019}, \cite{rudyak2022}, and \cite{scott2022}. However, as shown in \cite[Corollary 3.25, p. 1635]{zapata2022}, C. A. I. Zapata and J.~Gonz\'{a}lez unify the previous notions. In the next statement, we present lower and upper bounds for the relative topological complexity of a map, which also gives a relationship between $\mathrm{TC}_g\hspace{.1mm}(f)$ and $\mathrm{TC}\hspace{.1mm}(f)$.

\begin{theorem}[Lower and upper bounds]\label{relatice-tc-relatice-sec}
For maps $f:Z\to Y$ and $g:X\to Y$, we have \[\max\{\mathrm{TC}\hspace{.1mm}(f)/\mathrm{sec}(g),\mathrm{sec}_{g}\hspace{.1mm}(f)\}\leq\mathrm{TC}_g\hspace{.1mm}(f)\leq\min\{\mathrm{TC}\hspace{.1mm}(f),\mathrm{sec}_{g}\hspace{.1mm}(f)\cdot\mathrm{TC}\hspace{.1mm}(Z)\}.\]
\end{theorem}
\begin{proof}
By definition together with Proposition~\ref{basic-properties-relative}, we have that
\begin{align*}
    \mathrm{TC}_g\hspace{.1mm}(f)&=\mathrm{sec}_{1_Z\times g}\hspace{.1mm}(e_{f})\\
    &\leq\mathrm{sec}\hspace{.1mm}(e_{f})\\
    &=\mathrm{TC}\hspace{.1mm}(f).
\end{align*}
On the other hand, by Proposition~\ref{composition}:
\begin{align*}
    \mathrm{TC}_g\hspace{.1mm}(f)&=\mathrm{sec}_{1_Z\times g}\hspace{.1mm}(e_{f})\\
    &\geq \mathrm{sec}_{1_Z\times g}\hspace{.1mm}(1_Z\times f);\\
    \mathrm{TC}_g\hspace{.1mm}(f)&=\mathrm{sec}_{1_Z\times g}\hspace{.1mm}(e_{f})\\
    &\leq \mathrm{sec}_{1_Z\times g}\hspace{.1mm}(1_Z\times f)\cdot \mathrm{sec}\hspace{.1mm}(e_2^Z)\\
    &= \mathrm{sec}_{1_Z\times g}\hspace{.1mm}(1_Z\times f)\cdot \mathrm{TC}\hspace{.1mm}(Z).\\
\end{align*}
 Note that $\mathrm{sec}_{1_Z\times g}\hspace{.1mm}(1_Z\times f)=\mathrm{sec}\hspace{.1mm}\left(1_Z\times g^\ast(f)\right)$. Moreover, by Proposition~\ref{secat-sec-produc}, $\mathrm{sec}\hspace{.1mm}\left(1_Z\times g^\ast(f)\right)=\mathrm{sec}\hspace{.1mm}\left(g^\ast(f)\right)$. Hence, $\mathrm{sec}_{1_Z\times g}\hspace{.1mm}(1_Z\times f)=\mathrm{sec}_g\hspace{.1mm}\left(f\right)$. Therefore, $\mathrm{sec}_{g}\hspace{.1mm}(f)\leq\mathrm{TC}_g\hspace{.1mm}(f)\leq \mathrm{sec}_{g}\hspace{.1mm}(f)\cdot\mathrm{TC}\hspace{.1mm}(Z)$.

\medskip Furthermore, by Proposition~\ref{basic-properties-relative}(3) together with Proposition~\ref{secat-sec-produc}, we have the inequality $\mathrm{TC}_g\hspace{.1mm}(f)\cdot \mathrm{sec}(g)\geq \mathrm{TC}\hspace{.1mm}(f)$.
\end{proof}

Theorem~\ref{relatice-tc-relatice-sec} implies the following statement.

\begin{corollary}\label{contractible-relative-tc}
    Let  $f:Z\to Y$ and $g:X\to Y$ be maps.
    \begin{enumerate}
        \item If $Z$ is contractible, then $\mathrm{TC}_g\hspace{.1mm}(f)=\mathrm{sec}_{g}\hspace{.1mm}(f)$.
        \item If $\mathrm{sec}_{g}\hspace{.1mm}(f)=1$, then $\mathrm{TC}_g\hspace{.1mm}(f)\leq\min\{\mathrm{TC}\hspace{.1mm}(f),\mathrm{TC}\hspace{.1mm}(Z)\}$.
        \item If $\mathrm{sec}\hspace{.1mm}(g)=1$, then $\mathrm{TC}_g\hspace{.1mm}(f)=\mathrm{TC}\hspace{.1mm}(f)$.
    \end{enumerate}
\end{corollary}

As a direct application of Corollary~\ref{contractible-relative-tc}(1) together with Theorem~\ref{characterizacao-ppf} we have that CP can be characterized in terms of relative TC.

\begin{proposition}\label{prop:cp-implies-value-relative-tc}
    Let $X,Y$ be spaces where $Y$ is Hausdorff and $F(Y,2)$ is contractible. Let $g:X\to Y$ be a map. Then $(X,Y;g)$ has CP if and only if  $\mathrm{TC}_g\hspace{.1mm}(\pi_{2,1}^Y)=2$. 
\end{proposition}

\medskip Let $\mathbb{R}^\infty$ be the space of infinite sequences of real numbers with only finitely many non-zero entries. Note that $F(\mathbb{R}^\infty,2)$ is contractible. Thus, Proposition~\ref{prop:cp-implies-value-relative-tc} implies the following example.

\begin{example}\label{exam:rpinfty}
  Let $X$ be a space, and $g:X\to \mathbb{R}^\infty$ be a map. One has $\mathrm{TC}_g\hspace{.1mm}(\pi_{2,1}^{\mathbb{R}^\infty})=2$ if and only if  $(X,\mathbb{R}^\infty;g)$ has CP.
\end{example}

Corollary~\ref{contractible-relative-tc}(1) together with Remark~\ref{rem:nonCP-1} implies the following example.

\begin{example}\label{exam:nonCP-relativeTC}
  Let $X,Y$ be spaces where $Y$ is Hausdorff and $F(Y,2)$ is contractible. Let $g:X\to Y$ be a map. One has $(X,Y;g)$ has not CP if and only if  $\mathrm{TC}_g\hspace{.1mm}(\pi_{2,1}^Y)=1$. 
\end{example}

Also, Theorem~\ref{relatice-tc-relatice-sec} together with Proposition~\ref{basic-properties-relative}(3) imply the following statement.

\begin{proposition}\label{secat-1-relative-tc}
   Let $f:Z\to Y$ and $g:X\to Y$ be maps. If $\mathrm{secat}\hspace{.1mm}(g)=1$, then $\mathrm{HTC}\hspace{.1mm}(f)\leq\mathrm{TC}_g\hspace{.1mm}(f)\leq\mathrm{TC}\hspace{.1mm}(f)$.   
\end{proposition}
\begin{proof}
One has \begin{align*}
   \mathrm{TC}_g\hspace{.1mm}(f)&\geq \mathrm{HTC}_g\hspace{.1mm}(f)\\
   &\geq \mathrm{HTC}\hspace{.1mm}(f)/\mathrm{secat}\hspace{.1mm}(g)\\
\end{align*} and thus $\mathrm{TC}_g\hspace{.1mm}(f)\geq \mathrm{HTC}\hspace{.1mm}(f)$.
\end{proof}

As a consequence of Proposition~\ref{secat-1-relative-tc} one has 

\begin{corollary}\label{fibration-secat-1-relatice-tc}
   Let $f:Z\to Y$ be a fibration and $g:X\to Y$ be a map such that $\mathrm{secat}\hspace{.1mm}(g)=1$. Then $\mathrm{TC}_g\hspace{.1mm}(f)=\mathrm{TC}\hspace{.1mm}(f)$.  
\end{corollary}

\begin{example}\label{exam:relativeTC-TC}
    Let $X$, $Y$ be spaces where $Y$ is a connected $m$-dimensional topological manifold without boundary, and $m\geq 2$. Let $g:X\to Y$ be a map with $\mathrm{secat}\hspace{.1mm}(g)=1$ (for example when $g$ is a homotopy equivalence). One has, by Corollary~\ref{fibration-secat-1-relatice-tc}, $\mathrm{TC}_g\hspace{.1mm}(\pi_{2,1}^Y)=\mathrm{TC}\hspace{.1mm}(\pi_{2,1}^Y)$.  
\end{example}

In addition, Theorem~\ref{relatice-tc-relatice-sec} together with Theorem~\ref{characterizacao-ppf} implies the following statement.

\begin{proposition}\label{cp-relative-tc}
 Let $X,Y$ be spaces where $Y$ is Hausdorff. Let $g:X\to Y$ be a map. If $(X,Y;g)$ has CP, then \[2\leq\mathrm{TC}_g(\pi_{2,1}^Y)\leq\min\{\mathrm{TC}(\pi_{2,1}^Y),2\mathrm{TC}(F(Y,2))\}.\]   
\end{proposition}

As a direct application of Proposition~\ref{cp-relative-tc} and Proposition~\ref{cohomology-cp} we have the following statement.

\begin{example}\label{exam:cohomology-inequalities}
  Let $m\geq 1$, $X$ be a space, and $g:X\to D^m$ be a map. If for a cohomology theory $h^\ast$ we have \begin{equation*}
 \left(g_|\right)^\ast(e^{m-1})\not\in\mathrm{Im}\left(i^\ast:h^\ast(X)\to h^\ast\left(g^{-1}(S^{m-1})\right)\right)   
\end{equation*} where $i:g^{-1}(S^{m-1})\hookrightarrow X$ is the inclusion map and $e^{m-1}\neq 0$ is an element of $h^{m-1}(S^{m-1})$, then \[2\leq\mathrm{TC}_g(\pi_{2,1}^{D^m})\leq\min\{\mathrm{TC}(\pi_{2,1}^{D^m}),2\mathrm{TC}(F(D^m,2))\}.\]   
\end{example}

Proposition~\ref{cp-relative-tc} also implies the following example.

\begin{example}\label{exam:cp-relativeTC-noncontractible}
  Let $X,Y$ be spaces where $Y$ is Hausdorff. Let $g:X\to Y$ be a map. If $(X,Y;g)$ has CP and $\mathrm{TC}_g(\pi_{2,1}^Y)\geq 3$, then $F(Y,2)$ is not contractible.   
\end{example}

 Furthermore, we have the following remark.

\begin{remark}\label{rem:relativeTC-cat}
  Let $X$, $Y$ be spaces where $Y$ is a connected $m$-dimensional topological manifold without boundary, and $m\geq 2$. Let $g:X\to Y$ be a map. One has, by Proposition~\ref{basic-properties}(6) together with Lemma~\ref{TFN}, \[\mathrm{TC}_g(\pi_{2,1}^Y)\leq \mathrm{cat}\left(F(Y,2)\times X\right).\]   
\end{remark}

Theorem~\ref{relatice-tc-relatice-sec} together with Remark~\ref{rem:relativeTC-cat} implies the following example.

\begin{example}\label{exam:cp-contractible-noncontractible}
   Let $X$, $Y$ be spaces where $Y$ is a connected $m$-dimensional topological manifold without boundary, and $m\geq 2$. Let $g:X\to Y$ be a map. Suppose that $(X,Y;g)$ has CP and $X$ is contractible. One has, 
   \begin{align*}
     2&=\mathrm{sec}_g(\pi_{2,1}^Y)\\
     &\leq \mathrm{TC}_g(\pi_{2,1}^Y)\\
     &\leq \mathrm{cat}\left(F(Y,2)\right).
   \end{align*} Then, $F(Y,2)$ is not contractible.  
\end{example}

By~\cite[Theorem 2.1, p. 29]{zapata2017non}, for any connected finite-dimensional topological manifold $M$, using the Fadell-Neuwirth fibration $\pi^M_{2,1}:F(M.2)\to M$, C. A. I. Zapata shows that $F(M,2)$ is not contractible.  Example~\ref{exam:cp-contractible-noncontractible} provides conditions for the non-contractibility of $F(M,2)$ in terms of the CP. Meanwhile, Example~\ref{exam:cp-relativeTC-noncontractible} gives conditions for the non-contractibility of $F(Y,2)$ in terms of the CP and values of $\mathrm{TC}_g(\pi_{2,1}^Y)$ for any Hausdorff space $Y$.  


\bibliographystyle{plain}

\end{document}